\documentclass[11pt,english]{article}
\usepackage[T1]{fontenc}
\usepackage[utf8]{luainputenc}
\usepackage{xcolor}
\usepackage{pdfcolmk}
\usepackage{amsthm}
\usepackage{amsmath}
\usepackage{stmaryrd}
\usepackage{stackrel}
\usepackage{setspace}
\PassOptionsToPackage{normalem}{ulem}
\usepackage{ulem}

\makeatletter

\providecolor{lyxadded}{rgb}{0,0,1}
\providecolor{lyxdeleted}{rgb}{1,0,0}

\numberwithin{figure}{section}
\numberwithin{equation}{section}
\numberwithin{table}{section}

  \input amssymb.sty
\usepackage{etoolbox}
\patchcmd{\thebibliography}{\section*}{\section}{}{}
\renewcommand{\refname}{REFERENCES}
\usepackage{algpseudocode} \usepackage{algorithm}
\usepackage{algpseudocode,algorithm}
\usepackage{lipsum}
\usepackage{caption}
\usepackage{graphics}

\usepackage{amsmath}
\usepackage{amssymb}
\usepackage[all]{xy}

\usepackage{epsfig}\usepackage{youngtab}

\newcommand{\ef}{\end{equation}}
\chardef\bslash=`\\ 

\newdir{:=}{{}}
\newcommand*\colvec[3][]{
    \begin{pmatrix}\ifx\relax#1\relax\else#1\\\fi#2\\#3\end{pmatrix}
}
\newtheorem{thm}{Theorem}[section]
\newtheorem*{thm*}{Theorem}

\newtheorem{cor}{Corollary}[thm]

\newtheorem{lem}{Lemma}[section]
\newtheorem*{lem*}{Lemma}

\newtheorem*{corl*}{Corollary}

\newtheorem{prop}{Proposition}[section]
\newtheorem{prop*}{Proposition}

\theoremstyle{definition}
\newtheorem{defn}{Definition}[section]
\newtheorem{examp}{Example}
\newtheorem*{examp*}{Example}
\newtheorem*{remark*}{Remark}
\newtheorem*{CC*}{Crossover Conjecture}
\newtheorem*{Note*}{Note}
\newtheorem*{defn*}{Definition}

 \theoremstyle{remark}
\newtheorem{remark}{Remark}[section]

 \renewcommand{\sectionmark}[1]{}

\makeatother

\usepackage{babel}
\begin{document}

\title{MUTATIONS AND POINTING FOR BRAUER TREE ALGEBRAS} 
\author{Mary Schaps and Zehavit Zvi}

\maketitle

 \begin{abstract}

Brauer tree algebras are important and fundamental blocks
in the modular representation theory of groups. In this research,
we present a combination of two main approaches to the tilting theory
of Brauer tree algebras.

The first approach  is the theory initiated by Rickard, providing
a direct link between the ordinary Brauer tree algebra and a particular
algebra called the Brauer star algebra. This approach was continued
by Schaps-Zakay with their theory of pointing the tree.

The second approach is  the theory developed by Aihara, relating
to  the sequence  of mutations from the ordinary
Brauer tree algebra to the star-algebra of the Brauer tree. Our main
purpose in this research is to combine these two approaches: 

We find an algorithm based on centers which are all terminal edges,
 for which we are able to obtain a tilting complex
constructed from irreducible complexes of length two \cite{SZ1}, which
is obtained from a sequence of mutations. 
For the algorithm given by Aihara in \cite{Ai}, which he showed in 
Cor. 2.6 gives a two term tree-to-star complex, we prove that Aihara's complex
 is obtained  from the corresponding
completely folded Rickard tree-to-star complex 
by a permutation of projectives.
\end{abstract}
\section{INTRODUCTION}

\noindent This work concerns Brauer tree algebras, a widely studied class of 
algebras of finite representation type which includes all  blocks of cyclic defect 
group in modular group representation
theory. In the last twenty-five years  modular group representation theory
 has been considerably enriched by the introduction of methods from algebra representation theory, 
foremost among them the theory of tilting complexes in \cite{R1}, which led to the Brou\'e 
conjecture that every block whose defect group is abelian is derived equivalent
to its Green correspondent.  

 ~ 

\noindent A block of cyclic defect group is a Brauer 
tree algebra and its Green correspondent is a Brauer star algebra.
Rickard proved \cite{R2} that every Brauer tree algebra has a tilting complex which 
makes it derived equivalent to the correspondng Brauer star algebra.  
Schaps-Zakay showed that the tilting complexes in the 
opposite direction can be constructed form irreducible 
project tive complexes of length two.  Since all the projectives
of the Brauer star algebra have the same simple form, 
the endomorphism algebra of such a tilting complex is 
easily constructed and the structure of the Brauer tree can 
be read off from the complex, making them an excellent introduction
to tilting theory.

~

\noindent There have been two main approaches to the tilting theory
of Brauer tree algebras: the all-at-once approach, going back to Rickard
\cite{R1} (later involving pointing), and the step-by-step approach
going back to König and Zimmermann {[}KZ1{]}, later formulated in terms
of mutations by Aihara \cite{Ai} and used recently by Chan \cite{Ch}
 and Zvonarevna \cite{Zv}. In this project we propose to combine
and compare the two approaches.

\section{DEFINITIONS AND NOTATION}

\subsection{\noindent DERIVED EQUIVALENCE}

\noindent We fix an abelian category $\mathcal C$ and we denote by  $Ch(\mathcal C)$
the category of cochain complexes of objects of $\mathcal C$. 
The differentials of the complex are 
morphisms $\left\{ d_{n}:C_{n}\rightarrow C_{n+1}\right\} $
satisfying  $d_{n+1} \circ d_{n}=0$. For any given complex $C^\bullet$, the 
cocycles  $Z^{n}(C^\bullet)$, coboundaries  $B^{n}(C^\bullet)$, and  cohomology modules 
$H^{n}(C^\bullet)=Z^{n}(C^\bullet)/B^{n}(C^\bullet)$ are defined as in such standard texts as \cite{W}.

~

\noindent If $C^\bullet$, $D^\bullet$\textit{ }are cochain complexes, a\textit{ }morphism
$f_\bullet:C^\bullet\rightarrow D^\bullet$\textit{ }is a cochain map, which is to say, a family
of morphisms $f_{n}:C_{n}\rightarrow D_{n}$ which
commute with $d$ in the sense that $f_{n+1}d_{n}=d_{n}f_{n}$. To avoid sounding
pedantic, we will usually refer to these simply as chain maps, even when they 
are mapping cochain complexes.

~

\noindent A morphism $f_\bullet:C ^\bullet \rightarrow D^\bullet$ between cochain complexes sends
coboundaries to coboundaries and cocycles to cocycles. Thus, it induces module
morphisms $f_{n}^{*}:H^{n}\left(C^\bullet \right)\rightarrow H^{n}(D^\bullet)$.

\noindent \begin{defn}

\noindent A morphism $f_\bullet:C^\bullet\rightarrow D^\bullet$ of cochain complexes is called
a \textit{quasi-isomorphism }if the induced maps $H^{n}\left(C^\bullet\right)\rightarrow H^{n}(D^\bullet)$
are all isomorphisms.

\noindent \end{defn}

\noindent \begin{defn}

\noindent A cochain complex $C$\textbf{ }is called \textit{bounded}\textbf{
}if almost all the $C^{n}$ are zero. The\textit{ }complex $C$\textbf{ }is \textit{bounded
below}\textbf{ }if there is a bound $a$\textbf{ }such that $C^{n}=0$
for all $n<a$. The cochain complexes which are partially or fully bounded
form full subcategories $Ch^{b},\thinspace Ch^{+},\thinspace Ch^{-}$of
$Ch(\mathcal C)$.

\noindent \end{defn}

\noindent \begin{defn}

\noindent The \textit{derived category} of an abelian category $\mathcal C$ is the category obtained from  $Ch(\mathcal C)$ by adding
formal inverses to all the quasi-isomorphisms between chan complexes.  It is called the \textit{bounded derived category} and denoted  $D^{b}(\mathcal C)$ if we consider only bounded complexes.
A \textit{derived equivalence} between two abelian categories is an equivalence of categories between their derived categories.
\noindent \end{defn}

\noindent  Consider a ring $R$ which is assumed to be associative but
not necessarily commutative, which in the sequel will typically be either
 a block of a group algebra over a field of characteristic $p$ dividing
 the order of the group or else a finite dimensional algebra
 over a field of arbitrary characteristic. The category $R-Mod$ of left $R$-modules 
is the abelian category of primary interest to us. For any such ring $R$, let $D^{b}(R)$ be the derived category
of bounded complexes of left $R$-modules.

\noindent \begin{defn}
\noindent We say that a cochain map $f:C\rightarrow D$ is \textit{null
homotopic} if there are maps $s_{n}:C^{n}\rightarrow D^{n+1}$ such
that $f=ds+sd$.

\noindent \end{defn}

\noindent \begin{defn}

\noindent Two cochain maps $f,g:C\rightarrow D$ are \textit{chain homotopic}
if their difference $f-g$ is null homotopic, in other words, if $f-g=sd+ds$
for some $s$. The maps $\left\{ s_{n}\right\} $ are called a \textit{homotopy}
 from $f$ to $g$. We say that $f:C\rightarrow D$ is a
\textit{homotopy equivalence} if there is a map $g:D\rightarrow C$
such that $g\circ f$ is chain homotopic to the identity map $id_{C}$
and $f\circ g$ is chain homotopic to the identity map $id_{D}$.
\end{defn}

~

\noindent \begin{defn}
 Let $f:C\rightarrow D$\textbf{ }be a map of cochain complexes.
The\textit{ mapping cone} of $f$ is a chain complex $Cone(f)$ whose
degree $n$ part is $C_{n+1}\bigoplus D_{n}$. The differential in
$Cone(f)$ is given by the formula:
\begin{alignat*}{1}
 & \ensuremath{d\left(c,b\right)=\left(-d_{C}\left(c\right),d_{D}\left(b\right)+f\left(c\right)\right)},\ensuremath{\quad c\in C_{n+1},~~b\in D_{n}}
\end{alignat*}

\end{defn}

\noindent  Any map $H^{n}\left(f\right):H^{n}\left(C\right)\rightarrow H^{n}\left(D\right)$
can be fit into a long exact sequence of cohomology groups by use of
the following device. There is a short exact sequence
\begin{alignat*}{1}
0 & \rightarrow D\rightarrow Cone\left(f\right)\overset{\delta}{\rightarrow}C\left[1\right]\rightarrow0
\end{alignat*}

\noindent of cochain complexes, where the left map sends $b$ to $(0,b)$,
and the right map sends $(c,b)$ to $\text{-}c$.

\subsection{\noindent TILTING COMPLEXES}

\noindent We now consider complexes $T$ of projective modules over an associative ring $R$. The
notation $T[n]$ denotes the complex which is isomorphic to $T$ as
a module but in which the gradation has been shifted $n$ places to
the left and the differential is the shift of the differential multiplied
by $(-1)^{n}$.

~

\noindent For any ring $R$, let $D^{b}(R)$ be the derived category
of bounded complexes of $R$-modules.

\noindent \begin{defn}\label{TC}Let $R$ be a Noetherian ring. A bounded complex
$T$ of finitely generated projective $R$-modules is called a \textit{tilting
complex} if:

\noindent \renewcommand{\labelenumi}{(\roman{enumi})}
\begin{enumerate}
\item \noindent Hom\textsf{$_{D^{b}\left(R\right)}\left(T,T[n]\right)=0$}
whenever\textsf{ $n\neq0$}.
\item \noindent For any indecomposable projective\textsf{ $P$}, define
the stalk complex to be the complex\textsf{ $P^{.}:0\rightarrow P\rightarrow0$.}
Then every such\textsf{ $P^{.}$ }is in the triangulated category
generated by the direct summands of direct sums of copies of\textsf{
$T.$}
\end{enumerate}
 \end{defn}

\noindent A complex $T$ satisfying only (i) is called
a \textit{partial tilting complex.}

\begin{singlespace}
\noindent \begin{defn}Fix an abelian category $\mathcal C$ and the
category of cochain  complexes $Ch(\mathcal{C})$. For two complexes
$X$\textbf{ }and $Y$\textbf{ }denote by $Z(X,Y)$ the set of morphisms
from $X$\textit{ }to $Y$\textbf{ }which are homotopic to zero. The
collection of all $Z(X,Y)$\textbf{ }forms a subgroup of Hom$_{C(\mathcal{C})}(X,Y)$.
Denote by $K(\mathcal{C})$\textbf{ }the\textbf{ }quotient category,
i.e.{\large{} }$K(\mathcal{C})$ is the category having the same objects
as $Ch(\mathcal{C})$ but with morphisms
 \begin{center} 
{Hom${{}_{K}}_{(\mathcal{C}}{_{)}(X,Y)=}$ Hom${{}_{Ch}}$$_{(\mathcal{C}}$${_{)}(X,Y)/Z(X,Y)}$,}
\end{center}
so that two homotopic maps are identified.
The quotient category $K(\mathcal{C})$\textbf{ }is called the  \textit{homotopy
category},  and a homotopy equivalence between 
complexes is an isomorphism in the homotopy category. For an abelian category ${\mathcal{C},}$ $K^{-}(\mathcal{C})$
is the homotopy category of right bounded complexes in ${\mathcal{C}}$,
and similarly one can define $K^{+}(\mathcal{C})$.\end{defn}
\end{singlespace}

~

\noindent The derived category is not an abelian category, but it is a triangulated category. 
The original theory of tilting concerned modules called tilting modules. 
Happel, in \cite{H} showed that if there was a tilting between two algebras  $\Lambda$ and $\Gamma$,
it induced a functor which was an equivalence of their derived bounded categories.
Rickard \cite{R1} then proved a converse when tilting modules were replaced by tilting complexes,
namely, that  there is a tilting complex $T$ over $\Lambda$ with endomorphism
ring End$_{D^{b}\left(\Lambda\right)}(T)^{op}\cong\Gamma$.

\subsection{\noindent BRAUER TREES}

From now on, we concentrate on a particular class of algebras, the Brauer
tree algebras. Even when the algebra which interests us is the block
of a group algebra, we will not use the actual block but rather its skeleton, 
a Morita equivalent algebra which is basic, so that the quotient by the radical
 is a direct sum of copies of the field. Suppose that in the original 
block, the dimension of the ith simple module was $m_i$.  
 When we have finished calculating the 
tilting complex $T=\bigoplus T_i$ using the skeleton, 
then we can recover the original block as opposite algebra of the endomorphism
ring of the tilting complex  $T'=\bigoplus T_i^{\oplus m_i}$.

\noindent \begin{defn}

\noindent Let $e$ and $m$ be natural numbers. A \textit{Brauer tree}
of type\textit{ }$(e,m)$\textit{ }is a finite tree $(V,\ensuremath{{\cal E}})$
where $V$ is the set of vertices, $\ensuremath{{\cal E}}$ is the
set of edges, $|\ensuremath{{\cal E}|}=e$ (hence $|V|=e+1$), together
with a cyclic ordering of the edges at each vertex and a designation
of an exceptional vertex which is assigned multiplicity $m$.

\noindent \end{defn}

\noindent The set of all edges of vertex $u$ is denoted by $\ensuremath{{\cal E}}(u)$.
By \textquotedbl{}cyclic ordering\textquotedbl{} we mean that for
each edge \textit{E} in $\ensuremath{{\cal E}}(u)$ there is a `next'
edge in $\ensuremath{{\cal E}}(u)$ and that edge has a next edge
in $\ensuremath{{\cal E}(u)}$ etc., until each edge of $u$ is counted
exactly once, in which case\textit{ E} is the next one. We note that
if \textit{E} and $F$ are the only edges of $u$ then $F$ is next
after \textit{E} and \textit{E} is next after $F$.

\noindent \smallskip{}

\noindent Every Brauer tree can be embedded in the plane in such a
way that the cyclic ordering on each $\ensuremath{{\cal E}}(u)$ is
the counterclockwise direction. The exceptional vertex will be
drawn as a black circle and the other vertices as open circles. Two important examples of Brauer trees are:

(i) The \textit{star} with the exceptional vertex in the center.

\noindent \begin{figure*}[h]
\centering
\includegraphics{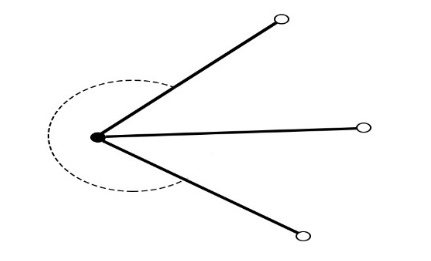}
\captionsetup{labelformat=empty}
\caption{Figure 1}
\end{figure*}

\noindent \smallskip{}

(ii) The \textit{linear tree}, which includes, for example, the Brauer trees of blocks of cyclic defect in the symmetric groups.

\noindent \smallskip{}

\noindent We relate Brauer trees to the structure of algebras. In
the definition we need to refer to uniserial modules, these being
modules in whose radical series each submodule has a simple top, where
the top of a module is the quotient by the radical.

\noindent \begin{defn}An algebra $A$ is called a \textit{Brauer
tree algebra} if there is a Brauer tree such that the indecomposable
projective $A$-modules can be described by the following algorithm:\renewcommand{\labelenumi}{(\roman{enumi})}
\begin{enumerate}
\item There is bijection between the edges of the tree and the isomorphism
classes of simple $A$-modules, i.e. each edge is labelled by the
corresponding isomorphism class. 
\item If $S$ is a simple $A$-module and $P_{S}$ is the corresponding
indecomposable projective $A$-module then $P_{S}\supseteq\text{rad}(P_{S})\supseteq\text{soc}(P_{S})\cong S$
and $\text{rad}(P_{S})/\text{soc}(P_{S})$ is a direct sum of one
or two uniserial modules corresponding to the two vertices of the edge, with composition factors determined 
by a clockwise circuit around the vertex. 
\end{enumerate}
\noindent \end{defn}

\noindent \begin{defn}Let $e$ and $m$ be natural numbers with $e>1$.
Let $K$ be any field containing a primitive $e$th root of unity
$\xi$. Let $\widehat{n}=em+1$. Let the cyclic group \({C_e} = \left\langle g \right\rangle \)
act on the truncated polynomial ring $A=K[x]/x\hat{^{n}}$, $g:x\longmapsto\xi x$.
The \textit{Brauer star algebra} of type $(e,m)$ is the skew group
algebra $
b=A[C_{e}],$
in which $g$ and $x$ obey the relation
$g^{-1}xg=\xi x.$
The algebra $b$ has $e$ distinct simple modules, corresponding to
the idempotents

\[
f_{i}=\frac{1}{e} \overset{e-1}{\underset{j=0}{\sum}}\xi^{-ij}g^{j},\quad i=1,..,e,
\]

\noindent and satisfying 
$f_{i}x=xf_{i+1}.$

~

\noindent The corresponding indecomposable projective left modules
are denoted by
$P_{i}=bf_{i},\quad i=0,\ldots,e-1.$
 Each $P_{i}$ is uniserial, and the projective cover of
rad$(P_{i+1})$ is $P_{i}$. We let $\left\{ x^{s}f_{i}\right\} _{s=0}^{em}$
be a basis for $P_{i}$, and define the following maps:
\begin{alignat*}{2}
\varepsilon_{i}:P_{i} & \rightarrow P_{i}, & \varepsilon_{i}\left(f_{i}\right)=x^{e}f_{i}\\
\\
\tilde{h}_{ij}:P_{i} & \rightarrow P_{j},\thinspace\thinspace & \tilde{h}_{ij}\left(f_{i}\right)=x^{k}f_{i} & ,~~k\equiv j-i(mod\,e),\quad0\le k<e.
\end{alignat*}

\noindent For $i\neq j$, we denote $\tilde{h}_{ij}$ by $h_{ij}$,
and for $i=j$ by $id_{i}$. For any $0\le\ensuremath{\ell}\leq m$
we call a map $\varepsilon_{j}^{\ell}\tilde{h}_{ij}\left(=\tilde{h}_{ij}\varepsilon_{i}^{\ell}\right)$
\textit{normal homogeneous} of degree $\ensuremath{\ell e+k}$, 

\noindent where 
\begin{align*}
k & \equiv j-i\left(mod\text{~}e\right),\quad0\le k<e.
\end{align*}
\end{defn}

\noindent \begin{corl*}[{\cite[Lemma 1.1]{SZ1}}]For the Brauer star algebra $b$, if $j\neq i$, and
$\left\{ j-i\right\} _{e}$ is the residue mod $e$, then there are
$m$ normal homogeneous maps $\varepsilon^\ell h_{ij}$,with degrees $s$  for $s=\left\{ j-i\right\} _{e}+\ell e\thinspace\thinspace\thinspace\thinspace\:\text{where\thinspace\thinspace}\ell=0,...,m-1$.
If $j=i,$ then there are $m+1$ normal homogeneous maps $\varepsilon^s:P_{i}\rightarrow P_{i}$,
for $s=0,e,2e,...,me.$\end{corl*}

\noindent \begin{defn}A cochain map $l_{\bullet}$ between $C^{\bullet}$ and $D^{\bullet}$
is called \textit{normal homogeneous} if each vertical map is normal
homogeneous.\end{defn}

\noindent \begin{defn}We call the homomorphism $\varepsilon_{i}^{m}:P_{i}\rightarrow P_{i}$
the \textit{socle map}, for the obvious reason that it maps the top
of $P_{i}$ into its socle \(\left\langle {{x^{em}}{f_i}} \right\rangle \).
\noindent \end{defn}

\noindent A partial tilting complex $T$\textit{ }for the Brauer star
algebra $b$ is called \textit{two-restricted}$(PTC_{2})$ if it is
a direct sum of shifts of the indecomposable complexes

\noindent \begin{center}
$\begin{array}{cccccccccccccc}
S_{i}: &  &  &  &  &  &  & 0 & \rightarrow & P_{i} & \rightarrow & 0\\
T_{ij}: &  &  &  &  &  &  & 0 & \rightarrow & P_{i} & \rightarrow & P_{j}& \rightarrow & 0,\quad i\neq j
 
\end{array}$
\par\end{center}

\noindent where the first nonzero  component of $S_{i}$
and $T_{ij}$ is in degree zero. The complexes $S_{i}[n]$ and $T_{ij}[n]$
are called \textit{elementary.} The map from $T_{ij}$ to $T_{ij}$
which is $\varepsilon_{i}^{m}$ on $P_{i}$ and zero on $P_{j}$ is
called the \textit{socle chain map}. It is chain homotopy equivalent
to the map which is zero on $P_{i}$ and $-\varepsilon_{j}^{m}$ on
$P_{j}$. Note that one can show that any indecomposable complex satisfying
Def \ref{TC}(i), which is nonzero in at most two degrees, is elementary 
and that a basis of the endomorphism ring of a tilting 
complex in $TC_2$ is given by the normal homogeneous maps \cite{SZ1}

~

~

\noindent \begin{defn}Let $B$ be a Brauer tree of type $(e,m)$.
A \textit{pointing} on $B$ is the choice, for each nonexceptional
vertex $u$, of a pair of edges $(i,j)$ which are adjacent in the
cyclic ordering at $u$. If there is only one edge $i$ at $u$, then
we take $(i,i)$ as the required pair. The tree $B$ together with
a pointing is called a \textit{pointed Brauer tree}.\end{defn}

\noindent \begin{remark}Recall that we have represented each Brauer
tree by a planar embedding and the cyclic ordering at each vertex
by counterclockwise ordering of the edges in the plane. We then represent
the pointing $(i,j)$ by placing a point in the sector between edge
$i$ and edge $j$ in a small neighborhood of $u$, as in Figure 2.\end{remark}

\noindent \begin{defn}Let $B$ be a Brauer tree with vertex set $V$.
The distance $d(u)$ of any vertex $u\in V$ from the exceptional
vertex $u_{0}$ is the number of edges in a minimal path from $u$
to $u_{0}$ (and hence in any path without backtracking, since the
graph is acyclic).\end{defn}

\noindent \begin{defn}Let $B$ be a Brauer tree with edge set $\ensuremath{{\cal E}}$.
An \textit{edge numbering} of $B$ is a Brauer tree with all its edges
numbered by $1,\ldots,e$. The \textit{vertex numbering} of $B$ is obtained
from the edge numbering by giving the same number as the edge to the
farthest vertex from the exceptional vertex on the edge. The exceptional
vertex is numbered as 0.\end{defn}

~

\noindent \begin{defn}  A\textit{
Green's walk} for a planar tree is a counterclockwise circuit of the
tree as if one were walking around the tree touching each edge with
the left hand. Each pointing and each choice of an initial brach determines an
 edge numbering  by starting at the exceptional vertex $v$
and taking a Green's walk around the tree which begins with the initial branch, 
and numbering the vertices and corresponding edges as $1, 2, 3,\dots, e$
as one come to the points. Figure
2 shows an example of such a vertex pointing.\end{defn}

~
\noindent \begin{defn}
 At any vertex besides the exceptional vertex, we will call the first edge that one
would meet on a Green's walk around the tree the \textit{ primary edge
}of the vertex, and the first edge one would meet on a 
reversed Green's walk will be called the \textit{coprimary edge}.
The pointing which puts the point between the entering edge and the 
primary edge at each vertex will be called the \textit{ordinary pointing}
and the pointing which puts the point between the entering vertex and
the coprimary edge will be called the  \textit{reversed pointing}.
The pointing which places the point first to the left and then to the right of 
the entering vertex, alternating as one goes out from the exceptinal vertex, will be called the 
\textit{left alternating pointing}, and there is dual which we will not need.
\end{defn} 

~

\noindent As described in \cite{SZ2}, each pointing determines a two-restricted star-to-tree tilting 
complex, in which the projectives of the tilting complex are from the 
Brauer star with the same $(e,m)$ and the opposite algebra of the 
endomorphism ring in the homotopy category is isomophic to the Brauer tree algebra of the tree 
which was pointed. The components $T_i$ of this star-to-tree
complex are stalk complexes for edges at the exceptional vertex
and complexes  $T_{ij}[-n_i]$ or $T_{ji}[-n_i+1]$, depending on 
whether the point is before or after $i$ in the cyclic ordering from the entering 
vertex $j$. The shifts are adjusted so that every $P_i$ appears in a unique degree $n_i$. A different pointing would give a different tilting complex 
with isomorphic endomorphism ring.
~
\noindent In the example in Figure 2, the vertex numbered $7$ corresponds to a partial two-restricted tilting complex,
ordered so that the vertical maps are generators of the endomorphism ring of 
the partial tilting complex. Note that the socle map is located just where 
the point is in the diagram, between $5$ and $8$.

\noindent \begin{figure*}[ht]
\centering
\includegraphics[scale=0.3]{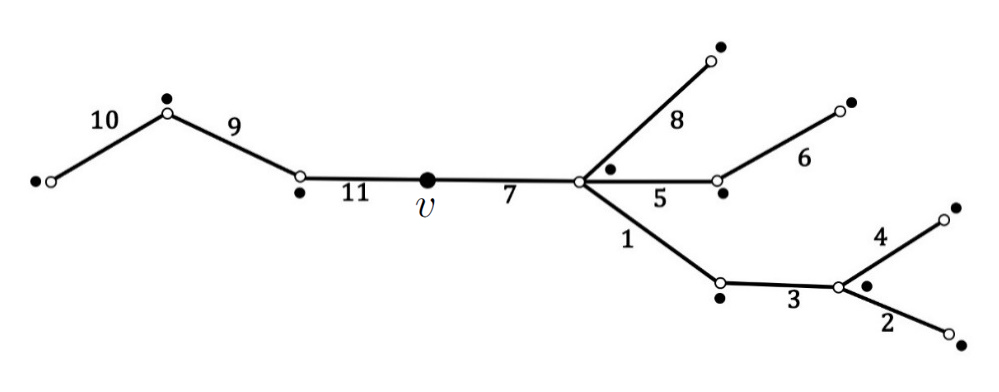}
\captionsetup{labelformat=empty}
\caption{Figure 2}
\end{figure*}

~

\noindent \begin{center}
$\begin{array}{cccccccccccccc}
T_{7}: &  &  &  &  &  &  & 0 & \rightarrow & P_{7} & \rightarrow & 0\\
 &  &  &  &  &  &  & & & id  \downarrow \\
T_{1}: &  &  &  &  &   0 & \rightarrow & P_{1} & \rightarrow  &  P_{7} & \rightarrow & 0\\
 &  &  &  &  &  &  & h_{15} \downarrow & & id \downarrow\\
T_{5}: &  &  &  &  & 0 & \rightarrow & P_{5} & \rightarrow  &  P_{7} & \rightarrow & 0\\
 &  &  &  &  &  &  & & & \varepsilon^m \downarrow\\
T_{8}: &  &  &  &  &  &  & 0 & \rightarrow & P_{7} & \rightarrow & P_{8}& \rightarrow & 0\\
 &  &  &  &  &  &  & & & id  \downarrow
\end{array}$
\par\end{center}

\noindent \begin{defn}Consider a sequence $\{r_i\}_{i=1}^{l}$ of elements
of $\left\{ 1,...,e\right\} $. Set
\[
h=\widetilde{h}_{r_{l-1}r_{l}}\circ...\circ\widetilde{h}_{r_{1}r_{2}}=\varepsilon_{r_{l}}^{\alpha}\widetilde{h}_{r_{1}r_{l}}.
\]

\noindent Then the sequence is $\textit{short}$ if $\alpha=0$ and
$\textit{long}$ if $\alpha>0$. We generally represent the sequence
in the form $r_{1}\rightarrow r_{2}\rightarrow...\rightarrow r_{l}$.\end{defn}

\noindent \begin{examp}If $e=11$ as in the example above,
\begin{itemize}
\item $1\rightarrow 5\rightarrow 7$ is short
\item $5\rightarrow 1\rightarrow 7$ is long.
\end{itemize}
\noindent \end{examp}

\noindent There is a result from \cite{SZ1} showing that  
a chain map $\ell_\bullet:T_{ik} \rightarrow T_{jk}$  has the identity map at $P_k$
if  $i \rightarrow j \rightarrow k$ is short and is the socle map if
$i\rightarrow j \rightarrow k$ is long, and similarly for the dual
map from  $T_{ij}$ to $T_{ik}$.

\subsection{\noindent MUTATION}

\noindent It is, of course, possible to define tilting complexes between
two general Brauer tree algebras. Of particular importance are the
tilting mutations of {[}Ai{]}, which go back to work of Rickard \cite{R2} and
Okuyama \cite{O}, or alternatively, to Kauer \cite{K}. Let $A$ be a finite dimensional basic algebra, with projective modules $P_j$. To each $j$,
we can associate an idempotent $\tilde{e}_{j}$ with $1_{A}$=$\stackrel[j=1]{e}{\sum}$$\tilde{e}_{j}$.

\noindent \begin{defn}Fix an $i$ and define $e_{0}=\underset{j\neq i}{\sum}$$\tilde{e_{j}}$.
For any $j$$\ensuremath{{\cal \in E}}$ we define a complex by 

$T_{j}^{(i)}=\begin{cases}
\begin{array}{cccc}
(0th) & \, & (1st) & \,\\
P_{j} & \longrightarrow & 0 & \:\:j\neq i\\
Q_{i} & \overset{\pi_{i}}{\longrightarrow} & P_{i} & \:\:j=i
\end{array} & \,\end{cases}$

\noindent where $Q_{i}\overset{\pi_{i}}{\longrightarrow}P_{i}$ is
a minimal projective presentation of $\text{\ensuremath{\tilde{e}_{j}A}}/\tilde{e}_{j}Ae_{0}A$.
Now we define $T^{(i)}:=\varoplus_{j\ensuremath{{\cal \in E}}}T_{j}^{(i)}$.
The\textit{ mutation }$\mu_{i}^{+}$ of $A$ is $A'$ $\cong$ End$_{D^{b}(A)}$$T^{(i)}$.
We will also consider the dual variant, as in $[S].$

$T_{j}^{(i)}=\begin{cases}
\begin{array}{cccc}
(-1st) & \, & (0th) & \,\\
0 & \longrightarrow &P_{j} & \:\:j\neq i\\
P_{i} & \overset{\pi_{i}}{\longrightarrow} & Q_{i} & \:\:j=i
\end{array} & \,\end{cases}$

~

\noindent where $Q_{i}$ is the minimal injective hull of the
quotient of $P_{i}$ by the largest submodule containing only
components isomorphic to the simple module  $S_{i}$.  This injective hull will be a direct
sum of injective modules (which are also projective) whose irreducible
socles give the socles of this quotient. We will denote this 
by $\mu^{-} \cite{AI}$ (see, e.g., \cite{Zv} for more detail in the case of Brauer trees.)

\end{defn}

~

\noindent Since $A$ is a symmetric algebra, either version of the mutation will give a tilting complex. 
(A similar complex can be defined also if $A$ is not symmetric, but then we get 
a complex which is no longer a tilting complex.) 

~

\noindent Now let $A$ be a Brauer tree algebra.  Aihara showed in \cite{Ai} that there is a simple combinatorial
operation on edges $j$$\ensuremath{{\cal \in E}}$ which corresponds to the mutation:
The edge $j$ is detached from both of its endpoints, and reattached to the tree at the farther end 
of the edge which is next before it in cyclic ordering.  
If the edge $j$ is a leaf, then there is only one reattachment made.

~

\noindent   By dualizing of Aihara's main theorem, \cite{Ai} Theorem 2.2, the mutation $\mu^{-}$ would correspond to the dual version of Aihara's operation
on the Brauer tree, namely, reattaching to the farther end of the edge which is \textit{after} it in the 
cyclic ordering. The diagrams to demonstrate this can be found in \cite{Zv}.

~

\noindent \begin{lem}\label{Ai} The inverse functor to the functor $G^+$ given by a mutation $\mu_i^+$ is the functor
$G^-$ given by $\mu_i^-$. 
\end{lem}

\begin{proof} Let us assume that all the stalk complexes are in degree $0$. We let the projectives in the 
Brauer tree algebra $A''$ on which $\mu_i^+$ is acting be denoted by $P_j^{''}$, and the corresponding 
projective modules in the algebra $A'$ on which $\mu_i^-$ is acting be denoted by $P_j^{'}$.
Let $Q_i^{''}$  be the corresponding projective cover of the radical, and let $Q_i^{'}$ be
the corresponding injective hull of the socle quotient.  Because of the biserial property 
of projectives of Brauer tree algebras, module  $Q_i^{''}$ is the sum of one or 
two projectives, corresponding to the edges to which $i$ is reattached, 
and similarly the projective-injective  $Q_i^{'}$ is the direct sum 
of the corresponding two projectives.  The actions of our two functors 
on the projectives are given by

\noindent \begin{center}
$\begin{array}{cccccccccccccccc}
G^{+}\left(P_{j}^{'}\right): &  & &  &  &   & 0 & \rightarrow & P_{j}^{''} & \rightarrow & 0, j \neq i\\
G^{+}\left(P_{i}^{'}\right): &  & & &   &  & 0 & \rightarrow & Q_{i}^{''} & \rightarrow  & P_{i}^{''} & \rightarrow & 0\\
G^{-}\left(P_{j}^{''}\right): & & &  & &   & 0 & \rightarrow & P_{j}^{'} & \rightarrow & 0, j \neq i\\
G^{-}\left(P_{i}^{''}\right): & &  &  & 0 & \rightarrow & P_{i}^{'} & \rightarrow  & Q_{i}^{'} & \rightarrow & 0
\end{array}$
\par\end{center}
 
Since the projectives correspond for all $j \neq i$, all we need to show is that  that  the cone $Cone\left(G^{+}\left(P_{i}^{'}\right) \overset{l_\bullet}\rightarrow G^{+}\left(Q_{i}^{'}\right)\right)$, 
with $l_{0}$  given by identity maps between the two copies of  $Q_{i}^{'}$ , is homotopy equivalent to  $P_i^{''}$. 

~

\noindent First, we define chain maps $f_{\bullet}$ from $Cone(l_{\bullet})$
to $P_i^{''}$ and $g_{\bullet}$ from $P_i^{''}$ to $Cone(l_{\bullet})$.
\begin{flalign*}
f_{-1} & =0&f_{0} & =(\pi_{P_{i}^{''}})\\
g_{-1} & =0 &g_{0} & =(0,id_{P_{i}^{''}})
\end{flalign*}
The resulting diagram is as follows,

\noindent \begin{equation*}
\xymatrix@C=90px@R=30px{ 
 {Q_{i}^{''}} \ar@{->}[r]^ {(id_{Q_{i}^{''}},h^{''})} \ar@{->}[d]^{0}  
& {Q_i^{''} \oplus P_{i}^{''}} \ar@{->}[d]^{-h^{''}\circ \pi_{Q_{i}^{''}} +\pi_{P_{i}^{''}}}\\
 {0}\ar@{->}[d]^-{0} \ar@{->}[r]^-{0}
& {P_i^{''}}\ar@{->}[d]^ {(0,id_{P_{i}^{''}})}\\ 
 {Q_{i}^{''}} \ar@{->}[r]^-{(id_{Q_{i}^{''}},h^{''})}   
& {Q_i^{''} \oplus P_{i}^{''}}}
\end{equation*}

\noindent and a diagram chase will show that both  $f_{\bullet}$
and  $g_{\bullet}$  are chain maps. The composition from the stalk complex to itself is the identity.
The homotopy from the composition $h_\bullet = g_\bullet \circ f_\bullet$ to the identity is given by 
$T=(-\pi_Q)$.

~

\noindent Using functoriality, we pull the functor $G^{+}$ outside the cone.
Hitting it on the left  by  $(G^{+})^{-1}$ we get
\[
(G^{+})^{-1}(P_i^{''})=\left( Cone\left(P_{i}^{'} \rightarrow Q_{i}^{'}\right)\right)=G^{-}(P_i^{''})
\]
and the remaining stalk complexes all correspond, which gives the desired result.  
\end{proof}

~

\noindent In the same paper \cite{Ai}, in Corollary 2.6, Aihara gave an algorithm for reducing a Brauer tree to 
a Brauer star.  In essence, the algorithm consists in doing a mutation centered at an edge 
of distance one from the exceptional vertex, as long as such edges exist.  In terms of 
number of steps to the star, this class of algorithms is very efficient, requiring only 
$e-\ell$ steps, where $\ell$ is the number of branches at the exceptional vertex, since each
step creates a new branch at the exceptional vertex.

~

\noindent Any mutation on a symmetric algebra gives a tilting and produces another symmetric algebra.
Thus if we have a sequence of mutations, we get a derived equivalence and hence a tilting 
complex.  Furthermore, any mutation of Brauer trees produces a one-to-one correspondence
of edges.  Thus, if we have a sequence of mutations leading to the Brauer star, any of the 
natural counterclockwise numberings of the Brauer star will induce a numbering of the Brauer tree.
The subject of this paper is the relationship between the choice of reduction procedure and this
natural numbering.

\section{MUTATION REDUCTION}

Assume we are given a Brauer tree $G$, with multiplicity $m$.  If $m>1$, 
then there is a designated exceptional vertex $v$. For $m=1$, we assume that one of
the vertices has been chosen as the exceptional vertex $v$. Since our graph is a tree,
there is a well-defined distance of each vertex $u$ from $v$ given by 
counting the number of edges on the unique path connecting them.
If the edges of the tree are labelled, then each vertex can be given
 the same label as the first edge on this unique path.
 
~

\noindent \begin{defn}
A \textit{mutation reduction} is a mutation or  sequence of mutations such that the distance of 
each vertex from the exceptional vertex does not ever increase,
and such that at least one such distance actually decreases. A mutation
reduction which ends at the Brauer star is called \textit{complete}.
\end{defn}

~

\noindent \begin{lem}\label{branch} Assume we are given a Brauer tree.
 
\begin{enumerate}
\item A mutation which is a mutation reduction must be centered at a primary edge.
\item A mutation centered at a primary edge connected to an edge adjacent to the 
exceptional vertex is a mutation reduction.
\item After a complete mutation reduction, all the edges from a given branch 
form an interval around the Brauer star, and these intervals follow the counterclockwise ordering of the branches.
\end{enumerate}
\end{lem}

\begin{proof}
\begin{enumerate}
\item If the mutation is not centered at a primary edge, then the mutation reattaches the center at the far end of the edge before it in the cyclic ordering, which is at greater distance from the exceptional vertex, in contradiction to our assumption that we have a mutation reduction.
\item The only mutation which can change the branch structure under a mutation reduction is a 
mutation by a primary edge $w$ connected to an edge $u$  adjacent to the exceptional vertex.
The effect of such a mutation is to is to create a new branch by lopping off $w$ and  the subgraph $S$
 of all edges connected to the exceptional vertex through the center $w$  of the mutation.  
The original branch rooted at $u$ will now be replaced by two branches, one rooted at $w$
 and connected to $S$ at the vertex at the opposite end of the edge $t$ which was last in
 counterclockwise order at the vertex to which $w$ was originally connected.
The other branch will be rooted at $u$, will be changed only in that $w$ and $S$ were removed, 
 and will follow the branch rooted at $w$ immediately in the counterclockwise ordering at $v$.

~

\noindent It remains to show that this operation was actually a mutation reduction.  The edge $w$, 
once at distance $1$, is now at distance $0$.
The edge $t$, once at distance $2$, is now at distance $1$, and every edge originally connected
 to $t$ and thus connected to the exceptional vertex via three edges, $u, w, t$, is now connected to $v$ via $w$ 
and is therefore at distance two less than before. Finally, the remaining edges of $S$, now all connected
to the exceptional vertex via $w, t$ instead of $u,w$, remain at exactly the same distance that they had before. 
 
~

\item   In any complete mutation reduction, each branch is eventually split entirely into separate leaves 
attached to the the exceptional vertex.  However, since this is always done, as described above, by separating 
one branch into two adjacent branches with the same labels as the original branch, the end result is that 
all the edges of in the original branch correspond to an interval around the star.
\end{enumerate}
\end{proof}

\noindent We will examine  two different mutation reduction algorithms, one 
a version of the original algorithm given by Aihara \cite{Ai}and  the other  our own from \cite{Z}.

~

\noindent \uline{Aihara's Algorithm}\cite{Ai} 
\begin{enumerate}
\item Choose an initial branch.  

\item In a Green's walk starting at the root of the  the initial branch choose the first   primary edge
$w$ attached to  an edge adjacent to the exceptional vertex.  If the tree is not a star, there must
be such an edge. 

\item By Lemma \ref{branch}(1), the mutation centered on this edge $w$ is a mutation
reduction, and from the proof we see that it creates two adjacent branches from the original,
the first of which in counter-clockwise order is rooted at $w$.

\item If $w$ was on the initial branch, let the new initial branch be the new branch rooted at $w$, and otherwise
let the initial branch remain as before.  Begin again from 2.

\end{enumerate}

\noindent \uline{Algorithm Z}\cite{Z} 
\begin{enumerate}
\item Choose an initial branch.  Let $d>1$ be the maximal distance of a vertex from the exceptional vertex $v$.

\item In a Green's walk starting on the initial branch choose the first leaf at distance $d$, necessarily  a primary edge, as center, 
and perform a series of mutations centered on the edge with this label,
for as long as it remains a primary edge or until it reached the exceptional vertex. The distances of 
all other vertices from the exceptional vertex will be unchanged.

\item Choose the next leaf at distance $d$ and proceed as in the previous item.
\item When there are no leaves left at distance $d$, then we find the new maximal distance $d'$.  If $d'=1$ we are 
finished, and otherwise we set $d=d'$ begin again at 2.
\end{enumerate}

\noindent  In terms of number of steps, this is as inefficient as a mutation reduction algorithm can be, 
because at each step, only one edge has its distance reduced by one.

~

\noindent We now construct a numbering on a Brauer tree depending
on which algorithm we use, by starting with the interval coming from the chosen initial branch
and numbering the edges in order.  This numbering
will be called the \textit{natural numbering} corresponding to this
choice of initial branch.

\noindent \begin{figure*}[ht]
\centering
\includegraphics[scale=0.1]{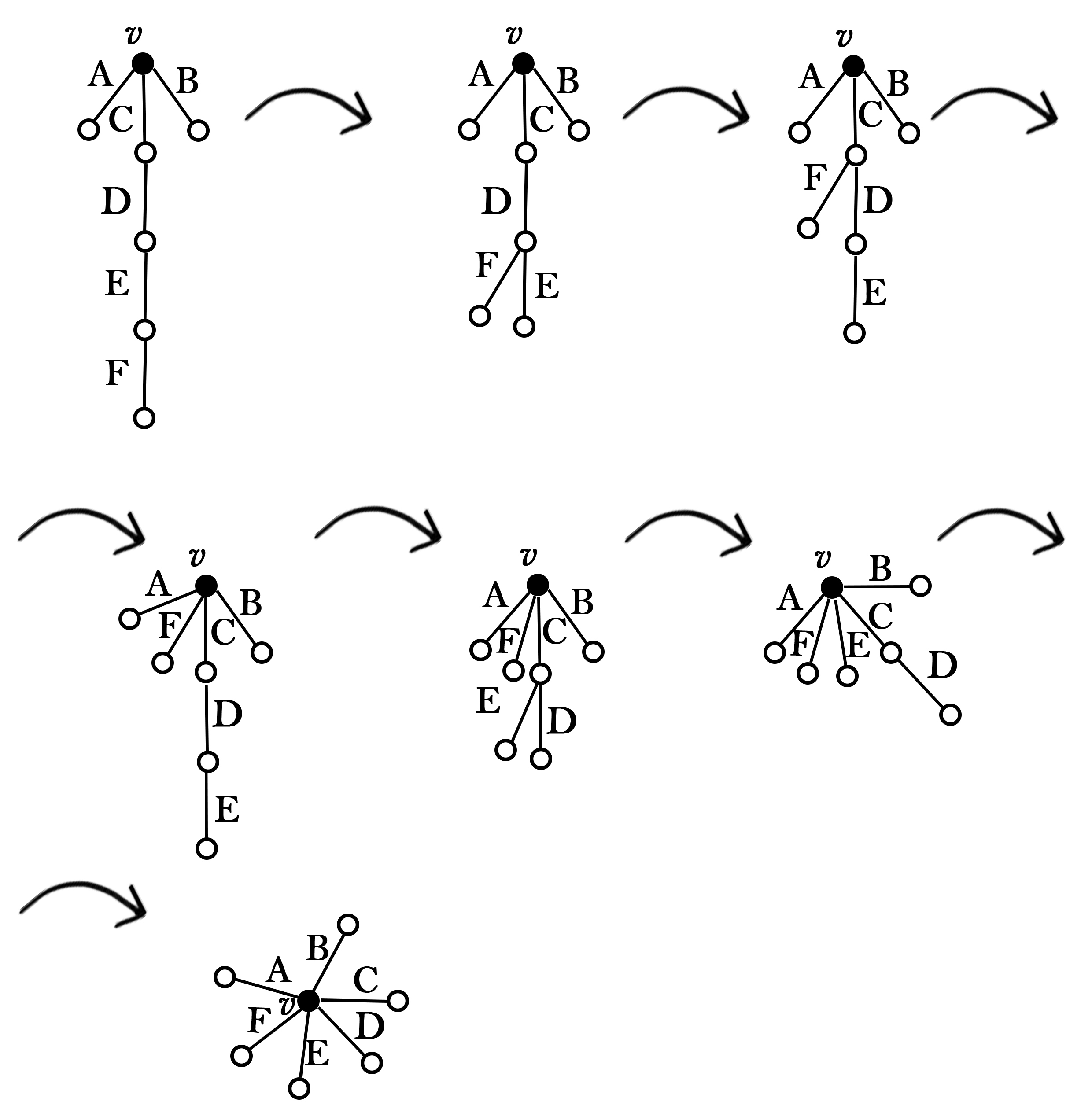}
\captionsetup{labelformat=empty}
\caption{Figure 3}
\end{figure*}

\noindent \begin{examp}

\noindent In Figure 3, we got to the Brauer star using Algorithm Z. Taking the largest 
branch as initial branch, we get as natural numbering

\noindent 
\[
A=6, B=5, C=4, D=3, E=2, F=1
\]
\end{examp}
~

\begin{figure*}[ht]
\centering
\includegraphics[scale=0.07]{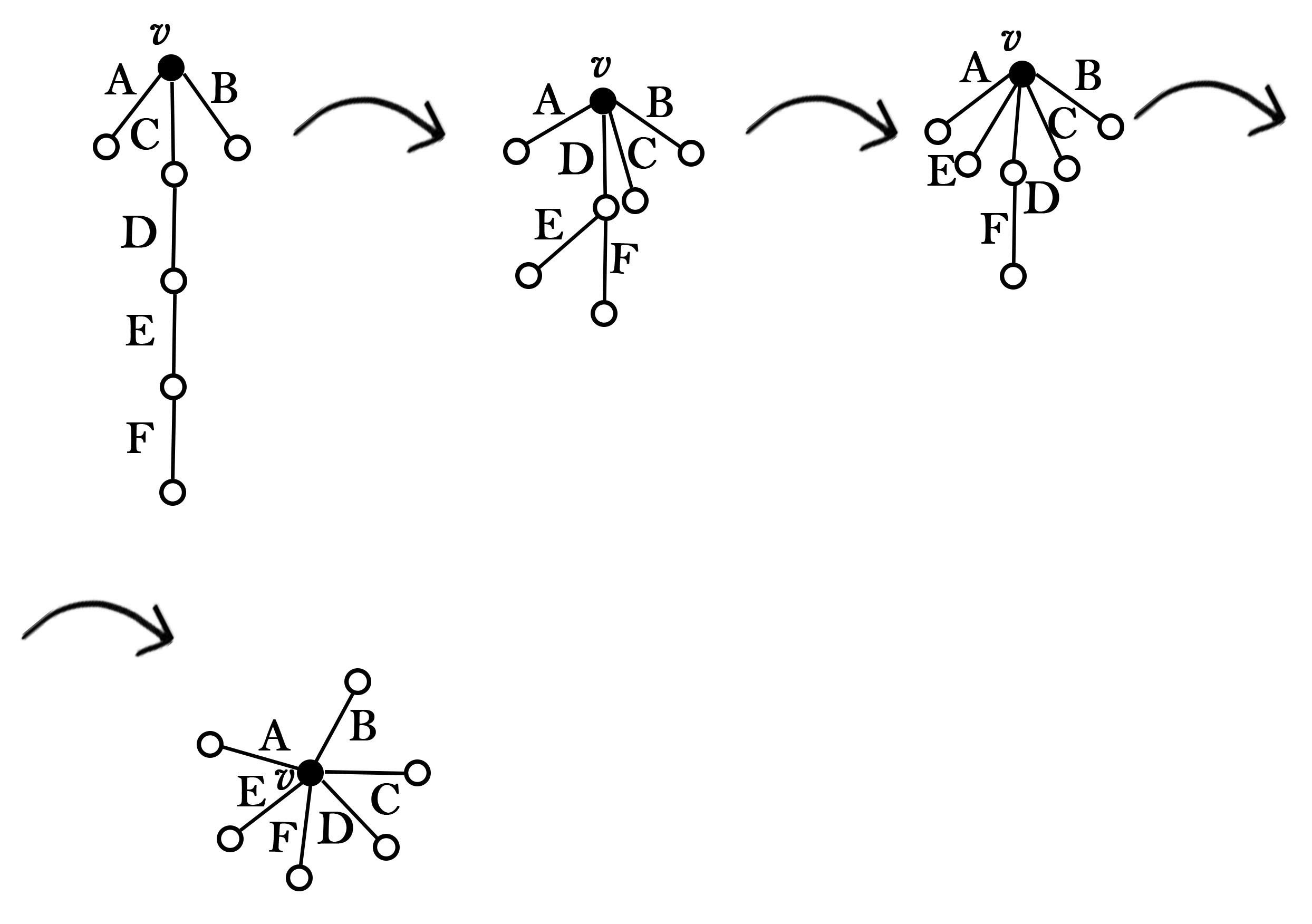}
\captionsetup{labelformat=empty}
\caption{Figure 4}
\end{figure*}

~
\noindent \begin{figure*}[!ht]
\centering
\includegraphics[scale=0.07]{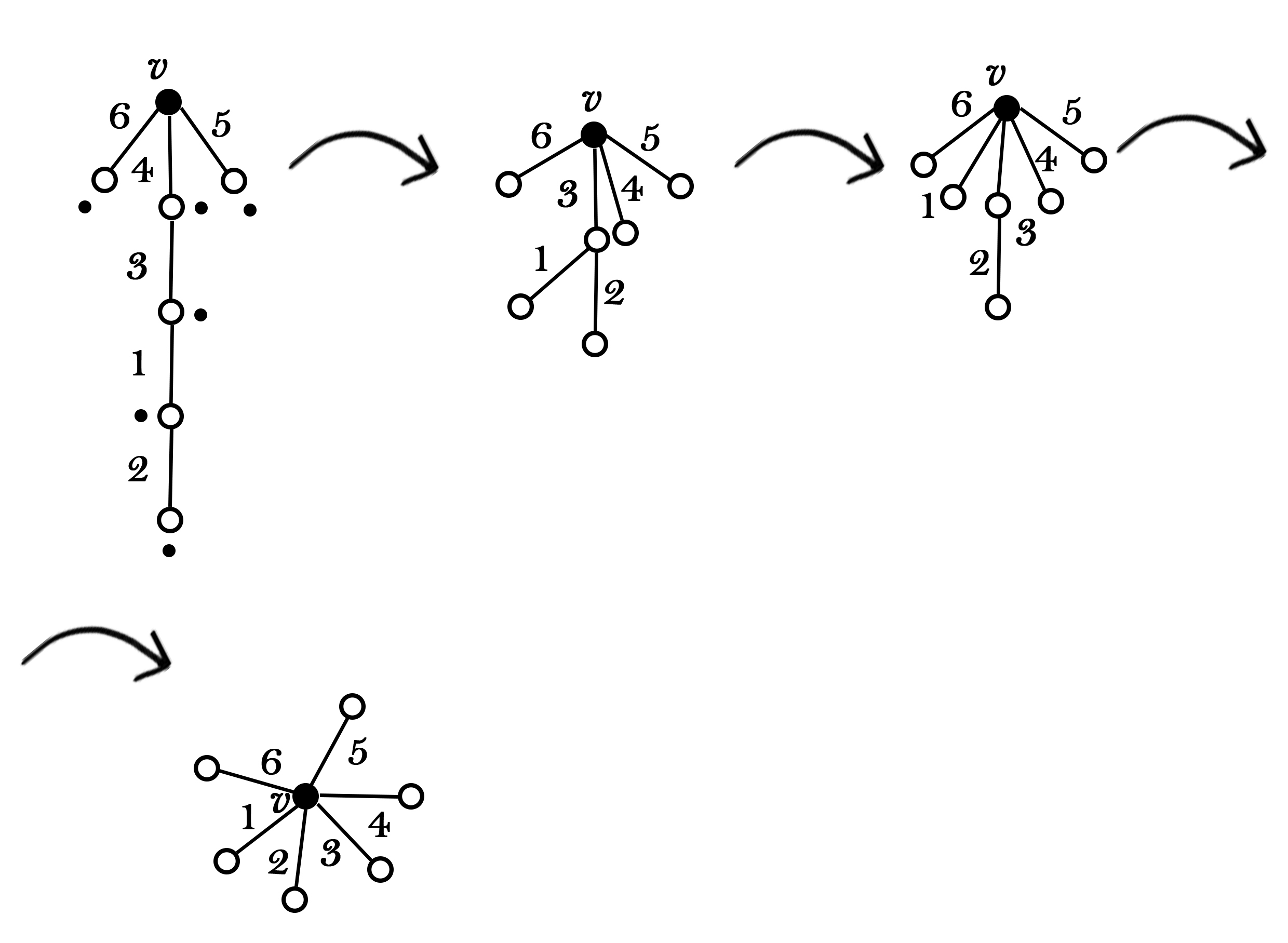}
\captionsetup{labelformat=empty}
\caption{Figure 5}
\end{figure*}
\noindent  \begin{examp}  Now, we follow Aihara's Algorithm  for the same original Brauer tree. 
 In Figure 4, using the same initial branch, we get a different natural numbering.

\noindent 
\[
A=6, B=5, C=4, D=3, E=1, F=2
\]

\end{examp}

~

\noindent  
The numbering gives a pointing and this pointing gives us a corresponding
star-to-tree tilting complex as described in \S 2. We will prove that 
from  Algorithm Z we get a tilting complex corresponding to  the natural numbering. From Aihara's Algorithm we obtain a tilting
complex which comes from a pointing but we will show that it is not usually  the pointing
corresponding to the natural numbering.

\section{MAIN THEOREM}

\noindent The basic step in Algorithm Z is to take a leaf $C$ which is a 
primary edge in a Brauer tree giving an algebra  $A^{''}$ and to do a mutation centered on this leaf, which will
be a mutation reduction to a tree whose algebra is  $A^{'}$.  The tilting complex of this 
mutation is expressed in terms of the projectives $P_i^{''}$ of $A^{''}$, given
by a functor $G:D^b(A^{'}) \rightarrow  D^b(A^{''})$.  Since $C$ is a 
primary edge attached to some edge $B$ which is closer to the exceptional vertex,
 the functor $G$ will act as the identity for every 
projective $P_i^{'}$ of $A^{'}$  except $P_C^{'}$, and for $C$ itself we will have
the projective cover of its radical, which is $P_B^{'}$ because $C$ is a primary edge adjacent to $B$.

\noindent \begin{center}
$\begin{array}{cccccccccccc}
G\left(P_{i}^{'}\right): &  &  &  &   & 0 & \rightarrow & P_{i}^{''} & \rightarrow & 0\\
G\left(P_{C}^{'}\right): &  &    &  &  & 0 & \rightarrow & P_{B}^{''} & \rightarrow  & P_{C}^{''} & \rightarrow & 0
\end{array}$
\par\end{center}

\noindent \begin{thm}\label{Alg} For any complete mutation reduction
whose centers are always leaves which are primary edges, the 
star-to-tree tilting complex of the composed mutations is the 
star-to-tree complex of the original tree with the reversed
pointing.
\end{thm}
\noindent \begin{proof} Let $A_\ell,A_{\ell-1},\dots,A_1$ be the Brauer tree algebras
 in the complete mutation reduction to the Brauer star algebra $A_0$. We number each of the 
corresponding Brauer trees by the natural numbering corresponding to this mutation reduction.
For each $k$ between $1$ and $\ell$, we let
\[F_k:D^b(A_0) \rightarrow D^b(A_k)
\] 
\[F_k^{-1}:D^b(A_k) \rightarrow D^b(A_0)
\]
be the functors obtained by composing the functors $G^{+}$ of the mutations $\mu^+$ and,
respectively,  the functors $G^{-}$ of the dual mutations $\mu^{-}$ in the opposite order.
We want to show that the star-to-tree tilting complex given by 
$F_\ell^{-1}$ is the star-to-tree complex given by the reversed 
pointing, from which it will follow by the results of \cite{RS} that 
the tilting complex inducing $F_\ell$ is Rickard's tree-to-star complex
for the same pointing.

~

\noindent Let us prove the theorem by induction on $\ell$.  If $\ell$ is $1$, then the Brauer tree has
only one edge $w$ not attached to the exceptional vertex $v$, but
rather to some $u$ attached to $v$.  If $w$ is numbered $i$ after mutation,
then $u$ will be numbered by $i+1$ since it comes after the new $w$ in the cyclic ordering of the star.
In the tilting complex of  $\mu_i^{-}$   we will have $Q_i=P_{i+1}$, and 
thus it coincides with the
star-to-tree complex of the reversed pointing.

~

\noindent Now assume that the theorem is true for $\ell-1$, so that 
$F_{\ell-1}^{-1}$ gives the star-to-tree complex of the reversed
pointing of the Brauer tree.  Let the $\{P_j^{'}\}$ be the projective
left modules of $A_{\ell-1}$ and let the $\{P_j^{''}\}$ be the projective
left modules of $A_{\ell}$. Let $i$ be the number of the center of 
the mutation in the Brauer tree of $A_{\ell-1}$ and let 
$j$ be the number of the next edge after it in the cyclic ordering.
By the rules for numbering edges at a vertex, we must have $j>i$.
Furthermore, since $i$ is a leaf, $P_i^{'}$ is uniserial, so the
injective hull $Q_i^{'}$ of $P_i^{'}/Soc(P_i^{'})$ is $P_j^{'}$.

\noindent \begin{center}
$\begin{array}{cccccccccccccccc}
G^{-}\left(P_{j}^{''}\right): & & &  & &   & 0 & \rightarrow & P_{j}^{'} & \rightarrow & 0, j \neq i\\
G^{-}\left(P_{i}^{''}\right): & &  &  & 0 & \rightarrow & P_{i}^{'} & \rightarrow  & P_{j}^{'} & \rightarrow & 0
\end{array}$
\par\end{center}

\noindent \textit{Case 1}. The edge $j$ is not attached to  the exceptional vertex:

\noindent \begin{figure*}[h]
\centering
\includegraphics[scale=0.21]{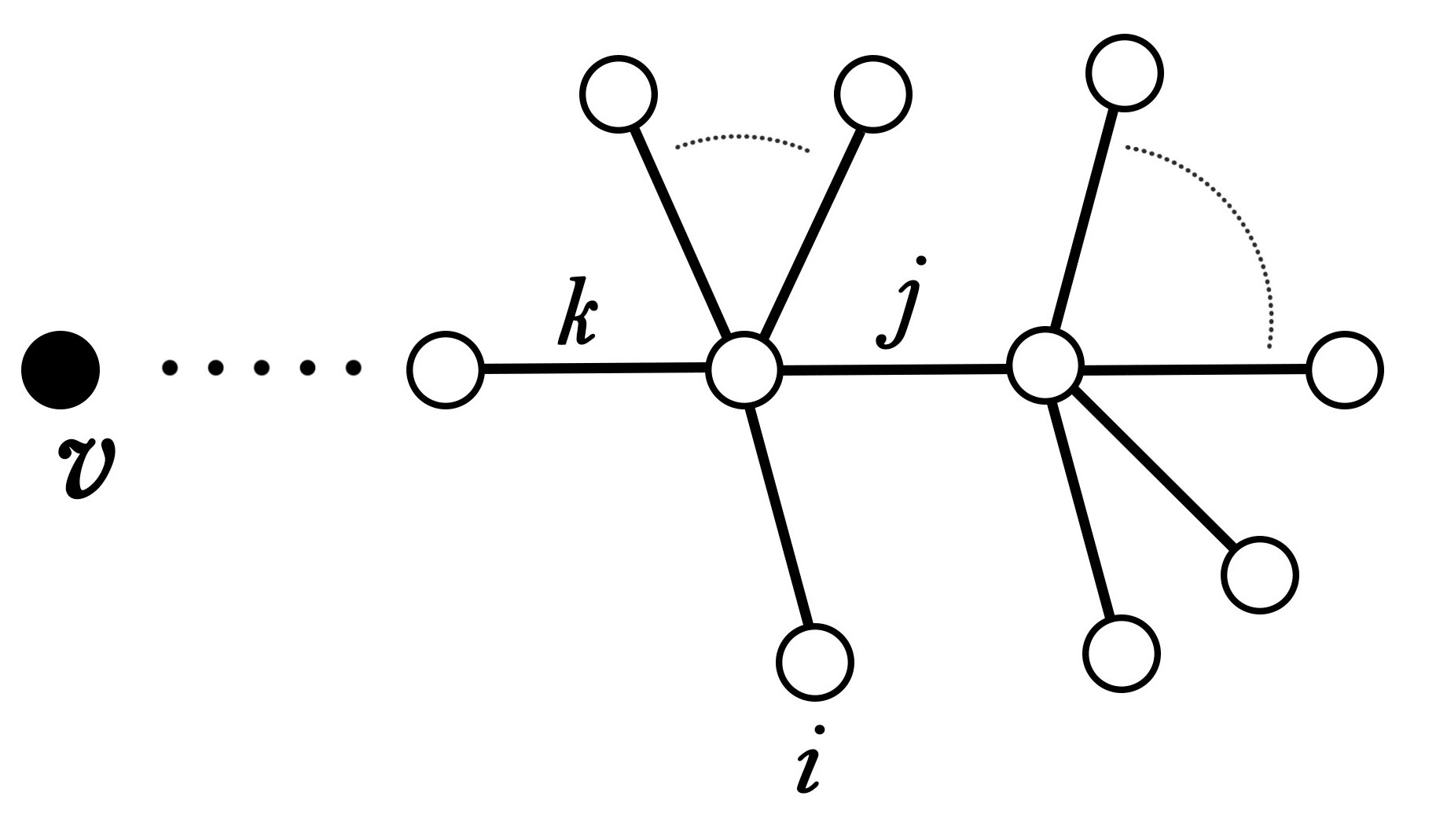}
\captionsetup{labelformat=empty}
\caption{Figure 6, $A_{\ell-1}$}
\end{figure*}

\noindent \begin{figure*}[ht]
\centering
\includegraphics[scale=0.21]{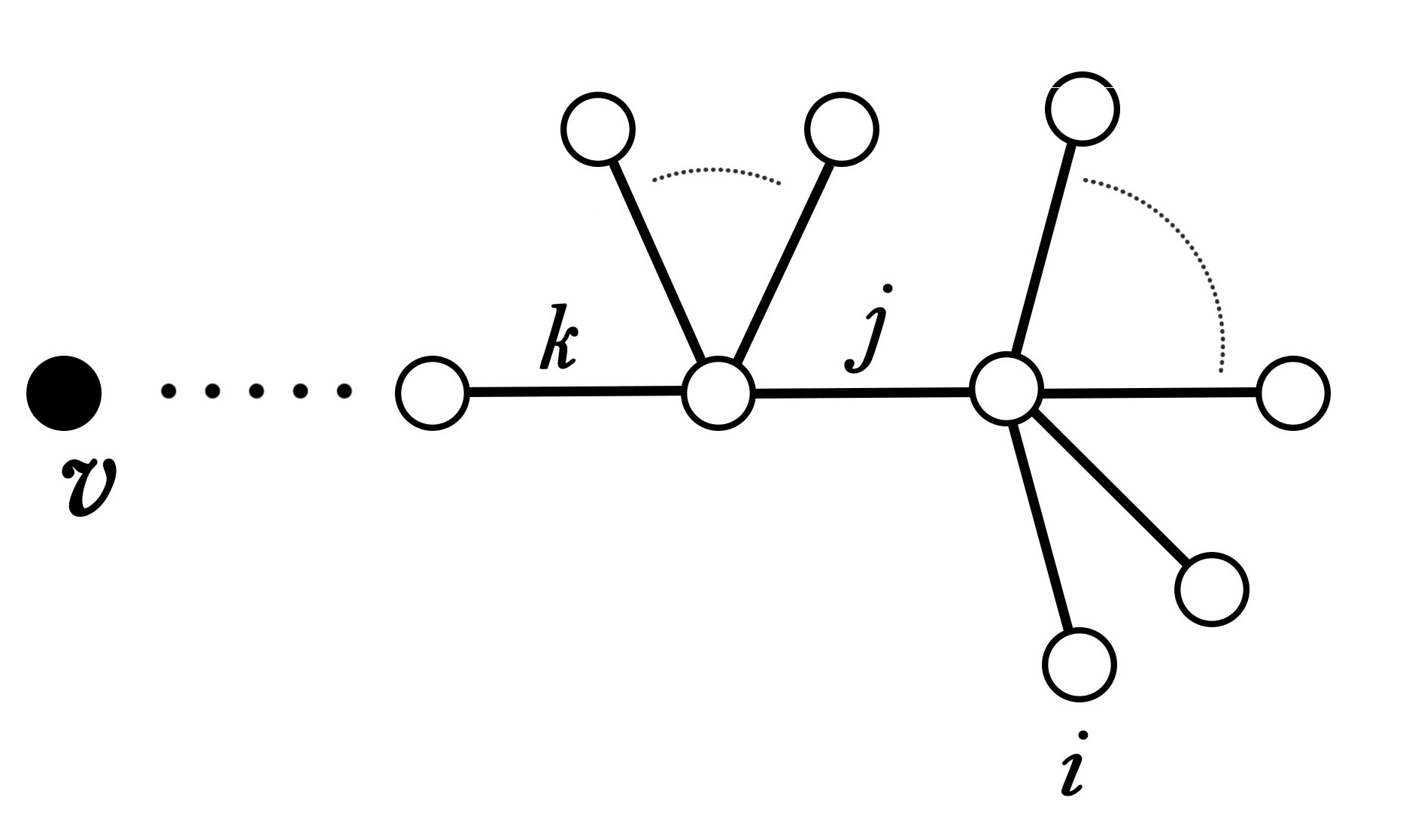}
\captionsetup{labelformat=empty}
\caption{Figure 7, $A_{\ell}$}
\end{figure*}
 
\noindent Let $k$ be the entering edge of the vertex at which $i$ and $j$ meet in the 
Brauer tree of $A_{\ell-1}$,as in Figure $6$, and assume that it is in degree $n_k$ in the tilting complex.
Then by the assumption of reversed pointing, we have $i<j<k$, and thus
$F_{\ell-1}^{-1}( P_i^{'}) =T_{ik}[n_k+1]$
 and $F_{\ell-1}^{-1}(P_j^{'})=T_{jk}[n_k+1]$.
 In this case we will get that the composition 
\begin{align*}
F_{\ell}^{-1}\left(P_{i}^{''}\right)&=F_{\ell-1}^{-1}\circ G^{-1}\left(P_{i}^{''}\right)\\
&=F_{\ell-1}^{-1}\left(Cone\left(P_{i}^{'}\rightarrow P_{j}^{'}\right)\right)\\
& =Cone\left(F_{\ell-1}^{-1}(P_{i}^{'})\rightarrow F_{\ell-1}^{-1}(P_{j}^{'})\right)
\end{align*}

\noindent \begin{equation*} = Cone \left(\!\!\!\!\!\!\!\!\!\!\!\!\!\!\!\! \vcenter{\xymatrix{ &0 \ar@{->}[r] & P_i \ar@{->}[r]^{h_{ik}} \ar@{->}[d]_{h_{ij}} & P_k \ar@{->}[r] \ar@{->}[d]^{id} & 0\\ &0 \ar@{->}[r] & P_j \ar@{->}[r]^{h_{jk}} & P_k \ar@{->}[r] & 0\\ }} \right) \end{equation*}

\noindent Denote the chain map in the cone by  $l^{\bullet}$.   We  compute $Cone(l^{\bullet})$ and  get:

\noindent \[ \xymatrix@C=60px{ & P_i \ar@{->}[r]^-{(-h_{ik}, h_{ij})} &P_k \oplus P_j \ar@{->}[r]^-{(\pi_{k}+h_{jk}\circ\pi_{j})} &P_k } \]

\noindent where $P_{k}$ is in degree 0. We want to show that $Cone(l^{\bullet})$
is homotopy equivalent to $T_{ij}[2]$, which is to say, $T_{ij}$
shifted so that the $P_{i}$ is in degree -2.

\noindent The chain maps $f_{\bullet}$ and $g_{\bullet}$ in the following diagram can be checked by composition or by diagram chasing.

\noindent \begin{equation*} 
\xymatrix@C=70px@R=30px{  {P_{i}} \ar@{->}[r]^-{(-h_{ik},h_{ij})} \ar@{->}[d]^-{id} 
& {P_k \oplus P_j}  \ar@{->}[r]^-{(\pi_{k}+h_{jk}\circ\pi_{j})}  \ar@{->}[d]^-{\pi_j} 
& {P_k} \ar@{->}[d]^-{0}\\ 
 {P_{i}} \ar@{->}[r]^-{h_{ij}} \ar@{->}[d]^-{id} & {P_{j}} \ar@{->}[r]^{} \ar@{->}[d]^-{(-h_{jk},id)} & {0} \ar@{->}[d]^-{0}\\  {P_{i}} \ar@{->}[r]^-{(-h_{ik},h_{ij})} & P_k \oplus P_j    \ar@{->}[r]^-{(\pi_{k}+h_{jk}\circ\pi_{j})}  & P_k  } \end{equation*}

The composition $f_\bullet \circ g_\bullet$ is the identity, so we need only prove that 
 the composition $h_{\bullet}=g_{\bullet}\circ f_{\bullet}$
 is homotopic to the identity of the mapping cone. We need to find $T_{1}:P_{k}\oplus P_{j}\rightarrow P_{i}$
, $T_{2}:P_{k}\rightarrow P_{k}\oplus P_{j}$ such that:

\noindent \[ \xymatrix@C=70px@R=50px{  P_i \ar@{->}[r]^-{(-h_{ik},h_{ij})}  \ar@<-0.5ex>@{->}[d]_{id} \ar@<0.5ex>@{->}[d]^-{id} & {P_k \oplus P_j} \ar@{->}[r]^-{(\pi_{k}+h_{jk}\circ\pi_{j})}  \ar@{->}[ld]_{T_1} \ar@<-0.5ex>@{->}[d]_-{(-h_{jk}\circ\pi_j,\pi_j)} \ar@<0.5ex>@{->}[d]^-{id} & {P_k} \ar@{->}[ld]_-{T_2} \ar@<-0.5ex>@{->}[d]_-{0} \ar@<0.5ex>@{->}[d]^-{id} \\  P_i \ar@{->}[r]^-{(-h_{ik},h_{ij})} & {P_k \oplus P_j} \ar@{->}[r]^-{(\pi_{k}+h_{jk}\circ\pi_{j})} & P_k \\ } \]

\noindent  To get the homotopy that we want, we choose $T_{1}=0$ and
 $T_{2}=(-id,0)$.

\noindent \textit{Case 2.} Near the exceptional vertex: $i$ is adjacent
to the exceptional vertex after doing the mutation. In Figure 8 we have the Brauer tree  for $A_{\ell}$
and in Figure 9 we have the relevant portion of the Brauer tree for $A_{\ell-1}$. 

\noindent \begin{figure*}[h]
\centering
\includegraphics[scale=0.15]{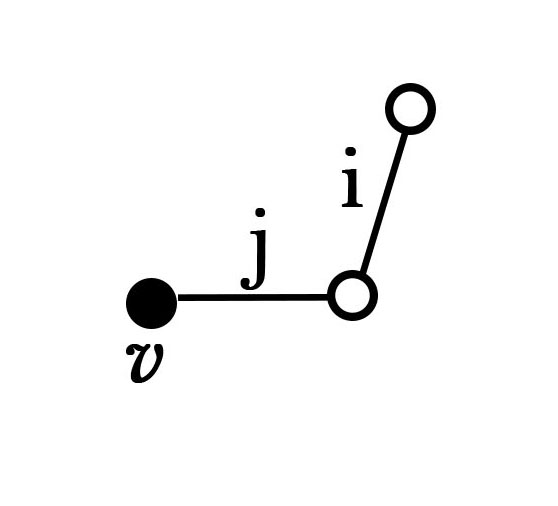}
\captionsetup{labelformat=empty}
\caption{Figure 8}
\end{figure*}

\noindent \begin{figure*}[h]
\centering
\includegraphics[scale=0.15]{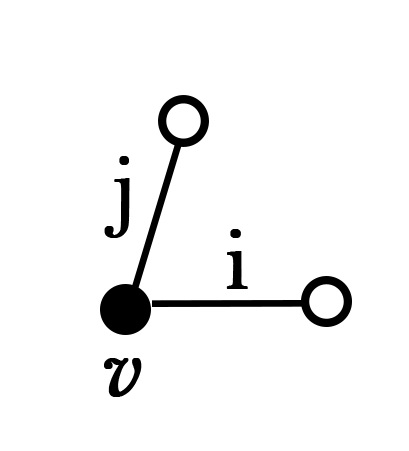}
\captionsetup{labelformat=empty}
\caption{Figure 9}
\end{figure*}

\noindent When we compute the tilting complex of the mutation, the
component of $P_{i}^{''}$ also is two-restricted. In this case we
will get that the composition \begin{equation*}F_{\ell-1}^{-1}\circ G^{-1}\left(P_{i}^{''}\right) = Cone \left(\!\!\!\!\!\!\!\!\!\!\!\!\!\!\!\! \vcenter{\xymatrix{ &0 \ar@{->}[r] & P_i \ar@{->}[d]_-{h_{ij}} \ar@{->}[r]  & 0\\ &0 \ar@{->}[r] & P_j \ar@{->}[r] & 0\\ }} \right) \end{equation*}

\noindent which is homotopy equivalent to $T_{ij}[1]$ , and
equal to $F_{\ell}^{-1}\left(P_{i}^{''}\right)$, which is precisely what we need for the 
star-to-tree tilting complex of the reversed pointing.

\end{proof}

\section{AIHARA'S ALGORITHM}

\noindent  In Cor. 2.6 of \cite{Ai}, Aihara shows that the tree-to-star functor obtained 
by composing the mutations is of length two.  We will compare
Aihara's functor with the completely folded two-term version of 
Rickard's tree-to-star functor given in \cite{RS}.

\noindent \begin{prop}\label{Alg2} We consider an arbitrary Brauer tree algebra.  Let $\sigma$ be the permutation of $1,\dots,e$ sending each number 
in the natural numbering of the tree by Aihara's Algorithm to the number of the corresponding edge in the left alternating numbering. Then
\begin{enumerate}
\item The star-to-tree complex obtained by composing the mutations of the algorithm in reverse order  can be obtained from the star-to-tree complex of the left alternating pointing 
 by 
permuting the rows by $\sigma$.
\item The tree-to-star complex corresponding the Aihara's Algorithm is the completely folded Rickard tree-to-star complex for the
left alternating pointing, except that the projectives are permuted by $\sigma$. 
\end{enumerate}
\end{prop}

\begin{proof} 
\begin{enumerate}
\item  We let  $\ell$ by the number of mutations in the complete mutation reduction. In the case $\ell = 1$, the center of the mutation is a leaf, so the natural numbering the the reversed numbering, and for a linear tree of length $2$, this coincides with the left alternating numbering, so the permutation is the identity.
We let $F^{-1}_\ell$ be the star-to-tree functor obtained by composing the inverse mutations, and let $H^{-1}_\ell$ be the star-to-tree functor given by the 
left alternating numbering, with $\sigma_\ell$ the permutation mapping the natural number of an edge to its number in the left alternating numbering.

We assume, by induction, that the proposition is true for $\ell-1$. Let the $P_i^{'}$ be the projectives for $\ell-1$ in the natural numbering, and  
let  $P_i^{''}$ be the projectives for $\ell$ in the natural numbering.  Let the $Q_i^{'}$ be the projectives for $\ell-1$ in the left alternating numbering, and  
let  $Q_i^{''}$ be the projectives for $\ell$ in the left alternatingl numbering. By our induction hypothesis, this means that for every $i, 1 \leq i \leq e$, 
\[
F^{-1}_{\ell-1}(P^{'}_i)=H^{-1}_{\ell-1}(Q^{'}_{\sigma_{\ell-1}(i)})
\]

For any branch, let the edge connected to the exceptional vertex be called the root. We are now going to perform an inverse mutation centered at a root $w$, which will join the branch with root $w$ to the next branch, with root $u$, where, as stated in Lemma \ref{branch}(3), we have $u>w$. Let $B_{\ell-1}$ be the Brauer tree before the branches rooted at $u$ and $w$ are joined, and let $B_\ell$ be the Brauer tree after they are joined. We let $t$ be the primary edge connected to $w$, and the same result shows that $w$ is the numerically highest number in the branch, so that  $w>t$. We note that  for the left alternating numbering the number of the root is also the highest in the branch, so that the permutation $\sigma$ always acts as the identity on roots.

 We compute the functor $G^{-1}$, the inverse of the mutation $\mu_w^{+}$  centered at $w$.

\noindent \begin{center}
$\begin{array}{cccccccccccccccc}
G^{-1}\left(P_{s}^{''}\right): & & &  & &   & 0 & \rightarrow & P_{s}^{'} & \rightarrow & 0, s \neq w\\
G^{-1}\left(P_{w}^{''}\right): & &  &  & 0 & \rightarrow & P_{w}^{'} & \rightarrow  & P_{u}^{'} \oplus  P_{t}^{'} & \rightarrow & 0
\end{array}$
\par\end{center}

Since $w$ and $u$ are both roots, we have

\noindent \begin{center}
$\begin{array}{cccccccccccccccc}
F_{\ell-1}^{-1}\left(P_{u}^{'}\right)=H_{\ell-1}^{-1}\left(P_{u}^{'}\right): & & &  & &   & 0 & \rightarrow & P_{u}^{} & \rightarrow & 0\\
F_{\ell-1}^{-1}\left(P_{w}^{'}\right)=H_{\ell-1}^{-1}\left(P_{w}^{'}\right):& & &  & &   & 0 & \rightarrow & P_{w}^{} & \rightarrow & 0
\end{array}$
\par\end{center}

It remains to calculate $F_{\ell-1}^{-1}\left(P_{t}^{'}\right)$. By the definition of the alternating pointing, we get $\sigma_{\ell-1}(t)=i$, where, $i$ is the lowest number in the branch rooted at $w$, and by the definition of the star-to-tree tilting complex of a given numbering, we have

\noindent \begin{center}
$\begin{array}{cccccccccccccccc}
 F_{\ell-1}^{-1}\left(P_{t}^{'}\right)= & &  &  & 0 & \rightarrow & P_{i}^{} & \rightarrow  & P_{w}^{}  & \rightarrow & 0
\end{array}$
\par\end{center}

We now calculate $F_{\ell}^{-1}$ as the composition $F_{\ell-1}^{-1}\circ G^{-1}$. W first make a general claim that if $i\rightarrow j \rightarrow k$ is short, then 
$ {P_{j} \oplus P_{i}}   \rightarrow  {P_k \oplus P_{j}}$ is homotopy equivalent to $T_{ik}[-1]$, where the map is given by $(-h_{jk}\circ\pi_{j},\pi_{j}+h_{ij}\circ\pi_{i})$.  The chain maps are obvious and the composition from  $T_{ik}$ to itself is the identity, so we need only find a homotopy from the opposite composition to the identity:

\noindent \begin{equation*}
\xymatrix@C=90px@R=30px{ 
 {P_{j} \oplus P_{i}} \ar@{->}[r]^-{(-h_{jk}\circ\pi_{j},\pi_{j}+h_{ij}\circ\pi_{i})} \ar@{->}[d]^{\pi_{i}}  
& {P_k \oplus P_{j}} \ar@{->}[d]^{\pi_{k}+h_{jk}\circ\pi_{j}}\\
 {P_i}\ar@{->}[d]^-{(-h_{ij},id)} \ar@{->}[r]^-{h_{ik}}
& {P_k}\ar@{->}[d]^-{(id,0)}  \\
 {P_{j} \oplus P_{i}} \ar@{->}[r]^-{(-h_{jk}\circ\pi_{j},\pi_{j}+h_{ij}\circ\pi_{i})} 
& {P_k \oplus P_{j}} }
\end{equation*}

 The needed homotopy is given by $T=(-\pi_{j},0)$.  With this result in hand, and noting that $i \rightarrow w \rightarrow u$ is short because $i \leq t <w <u$,  we make our calculation.

\begin{align*}
F_{l-1}^{-1}\circ G^{-1}\left(P_{w}^{''}\right)&=F_{l-1}^{-1}\left(Cone\left(P_{w}^{'}\rightarrow P_{u}^{'} \oplus P_{t}^{'}\right)\right) \\
&=Cone\left(F_{l-1}^{-1}(P_{w}^{'})\rightarrow F_{l-1}^{-1}(P_{u}^{'})\oplus F_{l-1}^{-1}(P_{t}^{'})\right) \\
&=Cone \left(\!\!\!\!\!\!\!\!\!\!\!\!\!\!\!\! \vcenter{\xymatrix{ & 0 \ar@{->}[r] & P_w \ar@{->}[r] \ar@{->}[d]^-{(h_{wu},id)}  & 0 \\ & P_t \ar@{->}[r]^-{(0,h_{tw})}  & P_u \oplus P_w \ar@{->}[r]&0 }} \right) \\
&= P_w \oplus P_i \to P_u \oplus P_w\\
& \equiv P_i \overset{h_{iu}} \longrightarrow P_u 
\end{align*}

The resulting star-to-tree complex is clearly in $TC_2$.  Because $G^{-1}$ is almost everywhere trivial, the only differences between $F_\ell$ and 
$F_{\ell-1}$ are:

\noindent \begin{center}
$\begin{array}{cccccccccccccccc}
 F_{\ell-1}^{-1}\left(P_{w}^{'}\right)= H_{\ell-1}^{-1}\left(Q_{w}^{'}\right)= & &  &  & & & 0 &  \rightarrow  & P_{w}^{}  & \rightarrow & 0\\
 F_{\ell}^{-1}\left(P_{w}^{''}\right)= H_{\ell}^{-1}\left(Q_{i}^{''}\right)= & &  &  & 0 & \rightarrow & P_{i}^{} & \rightarrow  & P_{u}^{}  & \rightarrow & 0\\
 F_{\ell-1}^{-1}\left(P_{t}^{'}\right)= H_{\ell-1}^{-1}\left(Q_{i}^{'}\right)= & &  &  & 0 & \rightarrow & P_{i}^{} & \rightarrow  & P_{w}^{}  & \rightarrow & 0\\
 F_{\ell}^{-1}\left(P_{t}^{''}\right)= H_{\ell}^{-1}\left(Q_{w}^{''}\right)= & &  &  & 0 & \rightarrow & P_{i}^{} & \rightarrow  & P_{w}^{}  & \rightarrow & 0\\
\end{array}$
\par\end{center}

Thus $\sigma_\ell$ is identical with $\sigma_{\ell-1}$ except on $w$ and $t$.  We have $\sigma_\ell(w)=\sigma_{\ell-1}(t)=i$ and 
 $\sigma_\ell(t)=\sigma_{\ell-1}(w)=w$.  It remains only to show the $H_\ell$ is indeed the star-to-tree functor for the left alternating pointing.  

\noindent In $B_{\ell-1}$, let $T$ be the collection of branches at the far end of $t$, let $W$ be the remaining branches connected to $w$, and let $U$ be the collection of
 branches at the far end of $u$. The interval in the natural numbering corresponding to $T$  is $[i,\dots,t-1]$, the interval corresponding to 
$W$ is $[t+1,\dots,w-1]$, and the interval corresponding to $U$ is $[w+1,\dots,u-1]$.  In $B_\ell$, $w$ is attached to $u$ before $U$, $t$ has become coprimary,
 with $W$ attached to its far end, and $T$ is now attached directly to $w$.  The distances of $U$ and of $T$ from the exceptional vertex $v$ remain as they were, and the distance of $W$ is increased by two, since it was originally attached directly to $w$, and now $u$, $w$ and $t$ intervene.  Since the intervals remain the same, and the alternating pointing remains the same, the left alternating numbering for each is the same, and thus $\sigma_\ell$ and $\sigma_{\ell-1}$ are identical on $U$, $W$ and $T$, and also on $u$, where both are fixed.

\noindent  Thus, for every vertex except $w$ and $t$, the left alternating pointing is exactly as it was. At the far end of $w$, the point is on the right, and thus the left alternating pointing assigns to $w$ the lowest number in the united branches, which was the lowest number in the branch with root $w$, which we called $i$.  In the united tree  $B_\ell$, the edge $t$ is coprimary going out from $w$ and has the point between it and $w$.  Since the edge $w$ has been assigned a low number, the alternating numbering assigns to $t$ the largest number in the branch originally rooted at $w$, which is $w$. The alternating numbering then gives to every other edge exactly the same number as before.  This proves 1.

\item By \cite{RS} the Rickard tree-to-star complex for the left alternating pointing is the inverse of the star-to-tree complex for the same pointing.  Since the star-to-tree for Aihara's algorithm differs from the star-to-tree for the left alternating pointing only in the order of the components, the Aihara complex differs from the Rickard complex only by the same  permutation of the projectives.

\end{enumerate}
\end{proof}

\begin{cor}
The permutation $\sigma$ is given by the cyclic ordering of edges on the vertices at non-zero even distance from the exceptional vertex.
\end{cor}
\begin{proof} As in the proof of the proposition, we do an induction on $\ell$, assuming that the result holds for $\ell-1$.  Thus in $\sigma_{\ell-1}$, the edges at the far end of $t$ are permuted according to the cyclic ordering at the vertex, from the primary edge $i$ through the numerically increasing starting edges of the branches in $T$, and  then to $t$, and finally from $t$ back to $i$.  In $B_\ell$, this vertex now has an extra edge.  The cyclic ordering goes from $i$ though the same sequence of starting edges in $T$, to the coprimary edge $t$, and finally to $w$.  This is precisely the change we documented in $\sigma_\ell$, where now $w$ goes to $i$ and $t$ to $w$, increasing the length of the cycle by one.

\end{proof}

\noindent \begin{examp}

We illustrate the above Proposition with a simple example, which will also demonstrate that the numbering we get from Proposition \ref{Alg2} will not, in general, be the natural numbering.

~
\noindent In Figure 10 we have a linear Brauer tree, which we reduce using Aihara's Algorithm to a Brauer star with $e=5$. 

~

\noindent \begin{figure*}[ht]
\centering
\includegraphics[scale=0.10]{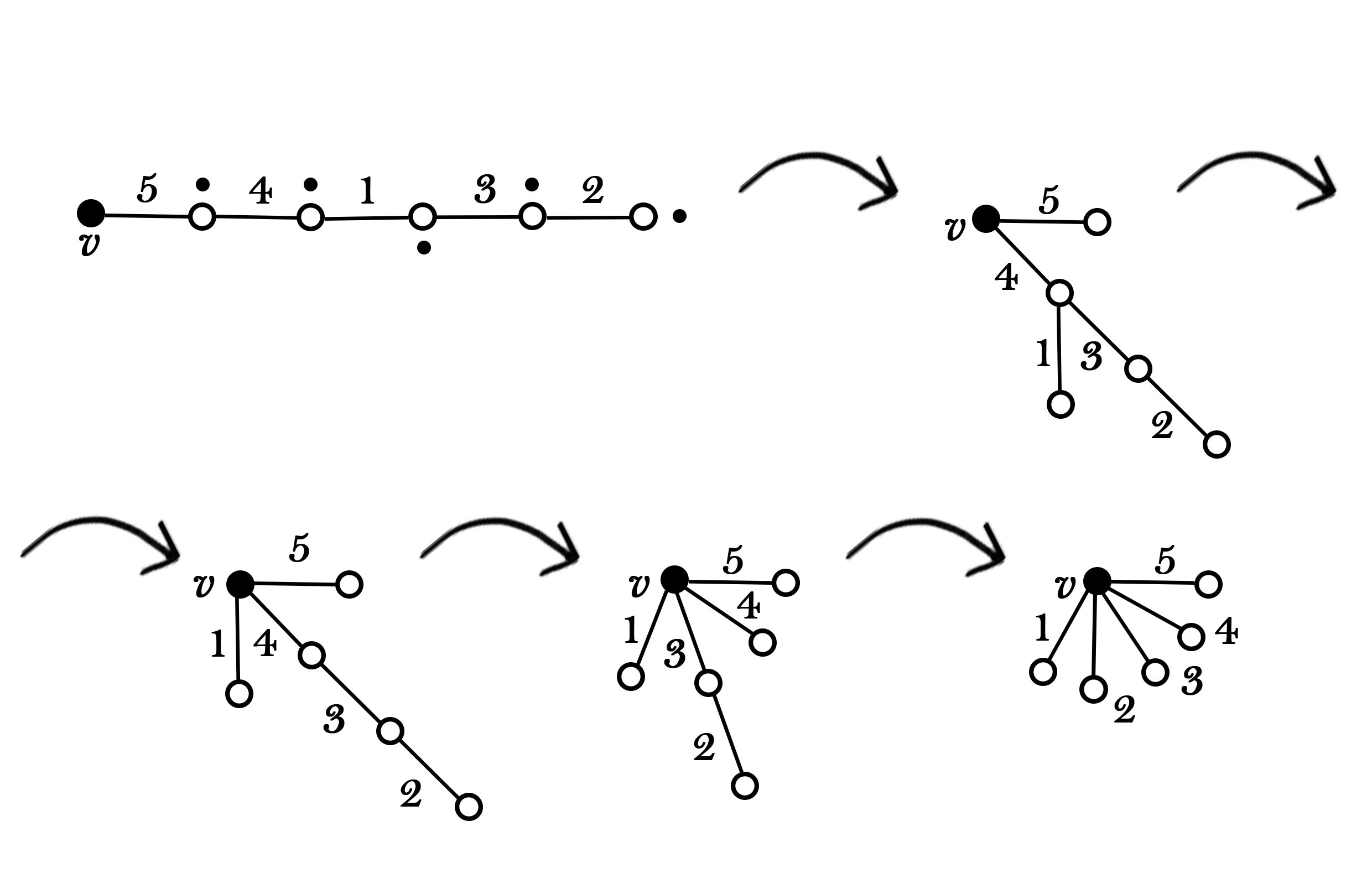}
\captionsetup{labelformat=empty}
\caption{Figure 10}
\end{figure*}

\noindent Now we compare this result with the composition of mutations as in Proposition \ref{Alg2}: 

\[
\begin{array}{cccccccccc}
F_{4}^{-1}(P_{1}^{''}): &  & 0 & \rightarrow & P_{1} & \rightarrow & P_{4} & \rightarrow & 0\\
F_{4}^{-1}(P_{2}^{''}):& & 0 & \rightarrow & P_{2} & \rightarrow & P_{3} & \rightarrow & 0\\
F_{4}^{-1}(P_{3}^{''}): &  & 0 & \rightarrow & P_{2} & \rightarrow & P_{4} & \rightarrow & 0\\
F_{4}^{-1}(P_{4}^{''}): &  &  0 & \rightarrow &   P_{1} & \rightarrow & P_{5} & \rightarrow & 0\\
F_{4}^{-1}(P_{5}^{''}): &  &  &  & 0 & \rightarrow & P_{5} & \rightarrow & 0
\end{array}
\]

\noindent This differs from  the folded star-to-tree complex for the pointed Brauer tree in Figure 11, 
constructed as in \cite{RS} by the ordering of the components of the tilting complex.  The image of 
$P_1^{''}$ and $P_4^{''}$ are exchanged, as are the images of $P_2^{''}$ and $P_3^{''}$.

\[
\begin{array}{cccccccccc}
H_{4}^{-1}(Q_{1}^{''}):  &  &   0 & \rightarrow & P_{1} & \rightarrow & P_{5} & \rightarrow & 0\\
H_{4}^{-1}(Q_{2}^{''}): &  & 0 & \rightarrow & P_{2} & \rightarrow & P_{4} & \rightarrow & 0\\
H_{4}^{-1}(Q_{3}^{''}): & & 0 & \rightarrow & P_{2} & \rightarrow & P_{3} & \rightarrow & 0\\
H_{4}^{-1}(Q_{4}^{''}):&  & 0 & \rightarrow & P_{1} & \rightarrow & P_{4} & \rightarrow & 0\\
H_{4}^{-1}(Q_{5}^{''}): &  &  &  & 0 & \rightarrow & P_{5} & \rightarrow & 0
\end{array}
\]

\noindent For completeness, we give  the corresponding Rickard tree-to-star complex.

\[
\begin{array}{ccccccccccc}
H_{4}(P_{1}): &  & & & 0 &\rightarrow  & Q_5^{''} & \rightarrow & Q_1^{''} & \rightarrow &  0\\
H_{4}(P_{2}):& &  &  & 0 & \rightarrow & Q_5^{''} \oplus  Q_{4}^{''} & \rightarrow &  Q_{1}^{''} \oplus Q_{2}^{''}  & \rightarrow  & 0\\
H_{4}(P_{3}):& &  &  & 0 & \rightarrow & Q_5^{''} \oplus  Q_{4}^{''} \oplus Q_{3}^{''} & \rightarrow &  Q_{1}^{''} \oplus Q_{2}^{''}  & \rightarrow  & 0\\
H_{4}(P_{4}):& &  &  & 0 & \rightarrow & Q_5^{''} \oplus  Q_{4}^{''} & \rightarrow &  Q_{1}^{''} & \rightarrow  & 0\\
H_{4}(P_{5}): &  &  &  & 0 & \rightarrow & Q_5^{''} & \rightarrow & 0
\end{array}
\]
\[
\begin{array}{ccccccccccc}
F_{4}(P_{1}):& &  &  & 0 & \rightarrow &  P_{5}^{''} & \rightarrow &  P_{4}^{''}  & \rightarrow  & 0\\
F_{4}(P_{2}):& &  &  & 0 & \rightarrow & P_5^{''} \oplus P_1^{''}  & \rightarrow & P_{4}^{''} \oplus P_{3}^{''}  & \rightarrow  & 0\\
F_{4}(P_{3}):& &  &  & 0 & \rightarrow & P_5^{''} \oplus  P_{1}^{''} \oplus  P_{2}^{''} & \rightarrow &  P_{4}^{''} \oplus P_{3}^{''}  & \rightarrow  & 0\\
F_{4}(P_{4}): &  & & & 0 &\rightarrow  & P_5^{''} \oplus  P_{1}^{''}  & \rightarrow &  P_{4}^{''}   & \rightarrow & 0\\
F_{4}(P_{5}): &  &  &  & 0 & \rightarrow & P_5^{''} & \rightarrow & 0
\end{array}
\]

\noindent The functor can be
obtained from the Rickard tree-to-star given by the functor $H_4$ above by a permutation of projectives
 exchanging
$1$ with $4$ and  of $2$ with $3$.

\noindent \begin{center}
\begin{figure*}[h]
\centering
\includegraphics[scale=0.15]{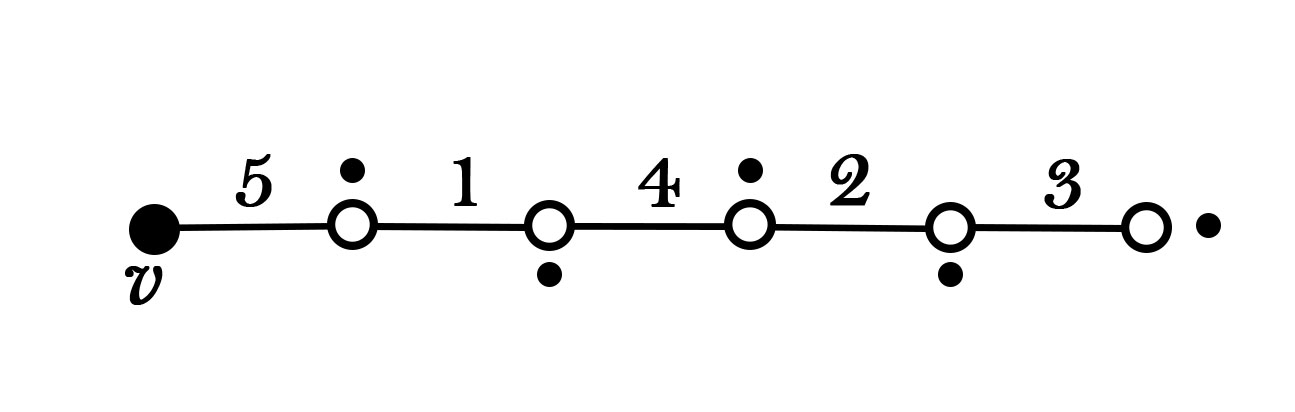}
\captionsetup{labelformat=empty}
\caption{Figure 11}
\end{figure*}
\par\end{center}
\end{examp}

\newpage{}\bibliographystyle{alpha}
\renewcommand{\refname}{REFERENCES}

\end{document}